\documentclass[12pt]{amsart}
\usepackage{amssymb,times,epsfig}

\makeatletter
\def\@strippedMR{}
\def\@scanforMR#1#2#3\endscan{
  \ifx#1M\ifx#2R\def\@strippedMR{#3}
  \else\def\@strippedMR{#1#2#3}
  \fi\fi}
\renewcommand\MR[1]{\relax\ifhmode\unskip\spacefactor3000 \space\fi
  \@scanforMR#1\endscan
  MR\MRhref{\@strippedMR}{\@strippedMR}}
\makeatother

\newtheorem*{Thm*}{Theorem}
\newtheorem{Thm}{Theorem}
\newtheorem{Cor}[Thm]{Corollary}
\newtheorem{Prop}[Thm]{Proposition}
\newtheorem{Lemma}[Thm]{Lemma}

\theoremstyle{definition}
\newtheorem{Defn}{Definition}

\newtheorem{Remark}[Defn]{Remark}
\newtheorem{Example}[Defn]{Example}
\newtheorem{Defn-Remark}[Defn]{Definition and Remark}

\newcommand{\mf}[1]{\mathbb{#1}}
\newcommand{\mc}[1]{\mathcal{#1}}
\newcommand{\mb}[1]{\mathbf{#1}}

\DeclareMathOperator{\NC}{\mathit{NC}}

\DeclareMathOperator{\Int}{\mathit{Int}}
\DeclareMathOperator{\Inner}{\mathit{Inner}}
\DeclareMathOperator{\Outer}{\mathit{Outer}}
\DeclareMathOperator{\Sing}{\mathit{Sing}}
\DeclareMathOperator{\SC}{\mathit{SC}}

\newcommand{\abs}[1]{\left\vert#1\right\vert}
\newcommand{\chf}[1]{\mathbf{1}_{#1}}
\newcommand{\set}[1]{\left\{#1\right\}}
\newcommand{\ip}[2]{\left \langle #1, #2 \right \rangle}
\newcommand{\state}[1]{\varphi \left[ #1 \right]}
\newcommand{\State}[1]{\Phi \left[ #1 \right]}
\renewcommand{\phi}{\varphi}

\newcommand{\Cum}[2]{R^{#1} \left[ #2 \right]}
\newcommand{\CumFun}[2]{R^{#1} \left( #2 \right)}

\newcommand{\eps}{\varepsilon}

\newcommand{\br}{\medskip\noindent}

\allowdisplaybreaks[1]

\title{Free evolution on algebras with two states}
\author[M.~Anshelevich]{Michael Anshelevich}
\thanks{This work was supported in part by NSF grant DMS-0613195}
\address{Department of Mathematics, Texas A\&M University, College Station, TX 77843-3368}
\email{manshel@math.tamu.edu}
\date{\today}

\begin{document}

\begin{abstract}
The key result in the paper concerns two transformations, $\Phi: (\rho, \psi) \mapsto \phi$ and $\mf{B}_{t}: \psi \mapsto \phi$, where $\rho, \psi, \phi$ are states on the algebra of non-commutative polynomials, or equivalently joint distributions of $d$-tuples of non-commuting operators. These transformations are related to free probability: if $\boxplus$ is the free convolution operation, and $\set{\rho_t}$ is a free convolution semigroup, we show that
\[
\State{\rho, \psi \boxplus \rho_t} = \mf{B}_{t} [\State{\rho, \psi}].
\]
The maps $\set{\mf{B}_t}$ were introduced by Belinschi and Nica as a semigroup of transformations such that $\mf{B}_1$ is the bijection between infinitely divisible distributions in Boolean and free probability theories. They showed that for $\gamma_t$ the free heat semigroup and $\Phi$ the Boolean version of the Kolmogorov representation for infinitely divisible measures,
\[
\State{\psi \boxplus \gamma_t} = \mf{B}_t[\State{\psi}].
\]
The more general map $\State{\rho, \psi}$ comes, not from free probability, but from the theory of two-state algebras, also called the conditionally free probability theory, introduced by Bo{\.z}ejko, Leinert, and Speicher. Orthogonality of the c-free versions of the Appell polynomials, investigated in \cite{AnsAppell3}, is closely related to the map $\Phi$. On the other hand, more general free Meixner families behave well under all the transformations above, and provide clues to their general behavior. Besides the evolution equation, other results include the positivity of the map $\State{\rho, \psi}$ and descriptions of its fixed points and range.
\end{abstract}

\maketitle

\section{Introduction}

\noindent
In a series of papers \cite{Belinschi-Nica-Eta,Belinschi-Nica-B_t,Belinschi-Nica-Free-BM}, Belinschi and Nica introduced and investigated a family of transformations $\mf{B}_t$. These transformations operate on measures, or more generally on ``non-commutative joint distributions'', i.e. states on non-commutative polynomial algebras. $\mf{B}_t$ are defined using a mixture of free and Boolean convolutions (see Section~\ref{Subsec:Convolutions}). They showed that these transformations form a semigroup, and for $t=1$, $\mf{B}_1$ is exactly the (Bercovici-Pata) bijection $\mf{B}$ between the infinitely divisible distributions in the free and Boolean probability theories. These transformations also have a remarkable relation to the free multiplicative convolution. On the other hand, Belinschi and Nica also proved that for a general state $\psi$,
\[
\State{\psi \boxplus \gamma_t} = \mf{B}_t[\State{\psi}],
\]
where $\boxplus$ is the free (additive) convolution and $\gamma_t$ is the free convolution semigroup of semicircular distributions, the free version of the heat semigroup. The map $\Phi$, which intertwines the actions of $\mf{B}_t$ and the semicircular evolution, can be described in the one-variable case as follows: for a measure $\nu$, $\State{\nu}$ is the measure whose Boolean cumulant generating function is
\[
\eta^{\State{\nu}}(z) = \int_{\mf{R}} \left( \frac{1}{1 - x z} - 1 - x z \right) \frac{1}{x^2} \,d\nu(x).
\]
The formula above is the Boolean version of the Kolmogorov representation (see Theorem~8.5 in \cite{Dur}). $\Phi$ is also a shift on the Jacobi parameters of the measure, see \cite{AnsAppell3}.

\br
Free probability \cite{VDN,Nica-Speicher-book} and Boolean probability theories are two of only three natural non-commutative probability theories (in addition to the usual probability theory; the third theory, monotone probability, also appears in the paper as an auxiliary device). Both the free and the Boolean setting are in fact particular cases of a more general construction for a space with \emph{two} expectations, or more precisely an algebra with two states $(\mc{A}, \phi, \psi)$. This theory is usually called the conditionally free or c-free theory. The objects in it are best described using their two-state cumulants $R^{\phi, \psi}$, which restrict to Boolean cumulants $\eta^\phi$ for $\psi = \delta_0$ and to free cumulants $R^\phi$ for $\psi = \phi$. It is not true that $R^{\phi, \psi}$ are always the free cumulants of some state. However, conversely, if $\psi$ is arbitrary and $\rho$ is a freely infinitely divisible state, then $R^\rho = R^{\phi, \psi}$ for some $\phi$. So we can define the map $\phi = \State{\rho, \psi}$. It is easy to see that $\State{\gamma, \cdot} = \Phi$, where $\gamma$ is the free product of standard semicircular distributions.  Moreover, we show that, now for general $\rho_t$,
\[
\State{\rho, \psi \boxplus \rho_t} = \mf{B}_t[\State{\rho, \psi}].
\]
This equation remains somewhat mysterious. However, we also exhibit an operator representation which provides a realization for it. This requires obtaining operator representations for all the ingredient maps, a result which may be of independent interest.

\br
Our original motivation for the study of the maps $\mf{B}_t$ and $\Phi$ came from their connection to the free Meixner distributions and states. Free Meixner distributions were originally defined as measures whose orthogonal polynomials have a special (free Sheffer) form, and in \cite{AnsBoolean} we showed that they have exactly the same characterization in terms of Boolean Sheffer families. In fact, as explained in \cite{AnsAppell3}, the natural home of the free Meixner distributions is in the c-free theory: they are those distributions $\phi$ which, for some $\psi$, have orthogonal c-free Appell (rather than the more general Sheffer) polynomials.

\br
It was noted by Belinschi and Nica and also in \cite{AnsBoolean} that the operations $\mf{B}_t$ take the class of free Meixner states to itself. We introduce a new two-parameter family $\set{\mf{B}_{\alpha, t}}$, which also form a semigroup. For $\alpha = 0$, these are exactly the Belinschi-Nica maps, while for $t=0$ these maps relate the Boolean convolution with the so-called Fermi convolution. These maps also preserve the free Meixner class, and moreover, in one variable using them the class can be generated from a single distribution $\frac{1}{2} (\delta_{-1} + \delta_1)$. In addition,
\[
\State{\rho, \psi \boxplus \rho_t \boxplus \delta_\alpha} = \mf{B}_{\alpha, t}[\State{\rho, \psi}],
\]
The actions on the free Meixner distributions of the operations $\boxplus$, $\mf{B}_{\alpha, t}$, and $\State{\rho, \cdot}$ can be computed explicitly, and lead to the more general results.

\br
The paper is organized as follows. Section~\ref{Section:Preliminaries} introduces the basic notions, in particular the two-state cumulants. Section~\ref{Section:Meixner} describes the appearance of the free Meixner distributions in the c-free theory, and defines $\State{\rho, \psi}$ and $\mf{B}_{\mb{a}, t}$. Section~\ref{Section:Transformations} contains the main results: semigroup property of $\mf{B}_{\mb{a}, t}$, positivity, fixed point, and image descriptions for $\State{\rho, \psi}$, and combinatorial and operator proofs of the equation which relates them.

\br
\textbf{Acknowledgements.}
I would like to thank Andu Nica and Serban Belinschi for a number of discussions which contributed to the development of this paper, and for explaining their work to me. I also thank Laura Matusevich for useful comments.

\section{Preliminaries}
\label{Section:Preliminaries}

\noindent
We will freely use the notions and notation from the Preliminaries section of \cite{AnsBoolean}; here we list the highlights.

\subsection{Polynomials and power series}
Let $\mf{C}\langle \mb{x} \rangle = \mf{C}\langle x_1, x_2, \ldots, x_d \rangle$ be all the polynomials with complex coefficients in $d$ non-com\-mu\-ting variables. They form a unital $\ast$-algebra.

\br
For a non-commutative power series $G$ in
\[
\mb{z} = (z_1, z_2, \ldots, z_d)
\]
define the left non-commutative partial derivative $D_i G$ by a linear extension of $D_i(1) = 0$,
\[
D_i z_{\vec{u}} = \delta_{i u(1)} z_{u(2)} \ldots z_{u(n)}.
\]

\subsection{Algebras and states}
Algebras $\mc{A}$ in this paper will always be complex $\ast$-algebras and, unless stated otherwise, unital. If the algebra is non-unital, one can always form its unitization $\mf{C} 1 \oplus \mc{A}$; if $\mc{A}$ was a $C^\ast$-algebra, its unitization can be made into one as well.

\br
Functionals $\mc{A} \rightarrow \mf{C}$ will always be linear, unital, and $\ast$-compatible. A state is a functional which in addition is positive definite, that is
\[
\state{X^\ast X} \geq 0
\]
(zero value for non-zero $X$ is allowed). A functional on $\mf{C} 1 \oplus \mc{A}$ is \emph{conditionally positive definite} if its restriction to $\mc{A}$ is positive definite. In particular, a functional on $\mf{C} \langle \mb{x} \rangle$ is conditionally positive definite if it is positive definite on polynomials without constant term.

\br
Most of the time we will be working with states on $\mf{C} \langle \mb{x} \rangle$ arising as joint distributions. For
\[
X_1, X_2, \ldots, X_d \in \mc{A}^{sa},
\]
their joint distribution with respect to $\psi$ is a state on $\mf{C} \langle \mb{x} \rangle$ determined by
\[
\state{P(\mb{x})} = \psi^{X_1, X_2, \ldots, X_d}\left[P(x_1, x_2, \ldots, x_d)\right] = \psi \left[P(X_1, X_2, \ldots, X_d)\right].
\]
The numbers $\state{x_{\vec{u}}}$ are the moments of $\phi$. More generally, for $d$ non-commuting indeterminates $\mb{z} = (z_1, \ldots, z_d)$, the series
\[
M(\mb{z}) = \sum_{\vec{u}} \state{x_{\vec{u}}} z_{\vec{u}}
\]
is the moment generating function of $\phi$.

\br
For a probability measure $\mu$ on $\mf{R}$ all of whose moments are finite, its monic orthogonal polynomials $\set{P_n}$ satisfy three-term recursion relations
\begin{equation}
\label{Three-term-recursion}
x P_n(x) = P_{n+1}(x) + \beta_n P_n(x) + \gamma_n P_{n-1}(x),
\end{equation}
with initial conditions $P_{-1} = 0$, $P_0 = 1$. We will call the parameter sequences
\[
(\beta_0, \beta_1, \beta_2, \ldots), (\gamma_1, \gamma_2, \gamma_3, \ldots)
\]
the Jacobi parameter sequences for $\mu$. Generalizations of such parameters for states with MOPS were found in \cite{AnsMulti-Sheffer}.

\subsection{Partitions}
We will denote the lattice of non-crossing partitions of $n$ elements by $\NC(n)$, the corresponding lattice of interval partitions by $\Int(n)$. A class $B \in \pi$ of a non-crossing partition is inner if for $j \in B$,
\[
\exists \ i \stackrel{\pi}{\sim} k \stackrel{\pi}{\not \sim} j: i < j < k,
\]
otherwise $B$ is outer. The collection of all the inner classes of $\pi$ will be denoted $\Inner(\pi)$, and similarly for $\Outer(\pi)$.

\br
From \cite{AnsFree-Meixner}, we take the following notation: for a set $V$,
\[
\NC'(V) = \set{\pi \in \NC(V): \min(V) \stackrel{\pi}{\sim} \max(V)}.
\]
For $\pi \in \NC'(V)$, we will denote $B^o(\pi)$ the unique outer class of $\pi$.

\br
For comparison, we also recall the notation from \cite{Belinschi-Nica-Eta}: for $\pi, \sigma \in \NC(n)$, denote $\pi \ll \sigma$ if
\[
\pi \leq \sigma, \text{ i.e. } i \stackrel{\pi}{\sim} j \Rightarrow i \stackrel{\sigma}{\sim} j
\]
but
\[
\forall B \in \sigma, \min(B) \stackrel{\pi}{\sim} \max(B).
\]
In particular, $\NC'(n) = \set{\pi \in \NC(n): \pi \ll \hat{1}_n}$.

\subsection{Cumulants}
For a state $\phi$, its Boolean cumulant generating function is defined by
\[
\eta^\phi(\mb{z}) = 1 - (1 + M^\phi(\mb{z}))^{-1},
\]
and its coefficients are the Boolean cumulants of $\phi$. They can also be expressed in terms of moments of $\phi$ using the lattice of interval partitions. Similarly, the free cumulant generating function of a state $\psi$ is defined by the implicit equation
\begin{equation}
\label{R-M-w}
M^\psi(\mb{w}) = \CumFun{\psi}{(1 + M^\psi(\mb{w})) \mb{w}},
\end{equation}
where
\[
(1 + M^\psi(\mb{w})) \mb{w} = \Bigl( (1 + M^\psi(\mb{w})) w_1, (1 + M^\psi(\mb{w})) w_2, \ldots, (1 + M^\psi(\mb{w})) w_d \Bigr).
\]
We will frequently, sometimes without comment, use the change of variables
\begin{equation}
\label{z-M-w}
z_i = (1 + M^\psi(\mb{w})) w_i , \qquad w_i = (1 + \CumFun{\psi}{\mb{z}})^{-1} z_i.
\end{equation}
The coefficients of $\CumFun{\psi}{\mb{z}}$ are the free cumulants of $\psi$, and can also be expressed in terms of the moments of $\psi$ using the lattice of non-crossing partitions:
\begin{equation}
\label{R-cumulant-moment}
\psi[x_1 \ldots x_n] = \sum_{\pi \in NC(n)} \prod_{B \in \pi} \Cum{\psi}{x_i : i \in B}.
\end{equation}

\subsection{Two-state cumulants}
The typical setting in this paper will be a triple $(\mc{A}, \phi, \psi)$, where $\mc{A}$ is an algebra and $\phi, \psi$ are functionals on it. By rotation, we can assume without loss of generality that $\phi$ is normalized to have zero means and identity covariance. Then no such assumptions can be made on $\psi$.

\br
We define the conditionally free cumulants of the pair $(\phi, \psi)$ via
\[
\state{x_1 \ldots x_n} = \sum_{\pi \in \NC(n)} \prod_{B \in \Outer(\pi)} \Cum{\phi, \psi}{\prod_{i \in B} x_i} \prod_{C \in \Inner(\pi)} \Cum{\psi}{\prod_{j \in C} x_j}.
\]
Their generating function is
\[
\CumFun{\phi, \psi}{\mb{z}} = \sum_{\vec{u}} \Cum{\phi, \psi}{x_{\vec{u}}} z_{\vec{u}}.
\]
Equivalently (up to changes of variables, this is Theorem~5.1 of \cite{BLS96}), we could have defined the conditionally free cumulant generating function via the condition
\begin{equation}
\label{eta-M-R}
\eta^\phi(\mb{w}) = (1 + M^\psi(\mb{w}))^{-1} \CumFun{\phi, \psi}{(1 + M^\psi(\mb{w})) \mb{w}}.
\end{equation}

\br
For elements $X_1, X_2, \ldots, X_n \in \mc{A}^{sa}$, we will denote their joint cumulants
\[
\Cum{\phi, \psi}{X_1, X_2, \ldots, X_n} = \Cum{\phi^{X_1, X_2, \ldots, X_n}, \psi^{X_1, X_2, \ldots, X_n}}{x_1, x_2, \ldots, x_n}
\]
to be the corresponding joint cumulants with respect to their joint distributions.

\begin{Defn}
\label{Defn:c-free}
Let $(\mc{A}, \phi, \psi)$ be an algebra with two states.
\begin{enumerate}
\item
Subalgebras $\mc{A}_1, \ldots, \mc{A}_d \subset \mc{A}$ are conditionally free, or c-free, with respect to $(\phi, \psi)$ if for any $n \geq 2$,
\[
a_i \in \mc{A}_{u(i)}, \quad i = 1, 2, \ldots, n, \qquad u(1) \neq u(2) \neq \ldots \neq u(n),
\]
the relation
\[
\psi[a_1] = \psi[a_2] = \ldots = \psi[a_n] = 0
\]
implies
\begin{equation}
\label{centered-product}
\state{a_1 a_2 \ldots a_n} = \state{a_1} \state{a_2} \ldots \state{a_n}.
\end{equation}
\item
The subalgebras are $(\phi | \psi)$ free if for $a_1, a_2, \ldots, a_n \in \bigcup_{j=1}^d \mc{A}_j$,
\[
\Cum{\phi, \psi}{a_1, a_2, \ldots, a_n} = 0
\]
unless all $a_i$ lie in the same subalgebra.
\end{enumerate}
\end{Defn}

\noindent
As pointed out in \cite{Boz-Bryc-Two-states}, these properties are \emph{not} equivalent, however they become equivalent under the extra requirement that the subalgebras are $\psi$-freely independent. In any case, throughout the paper we will be working with cumulants and will not actually encounter conditional freeness.

\begin{Example}
If $a, c$ are c-free from $b$, then (Lemma 2.1 of \cite{BLS96})
\[
\state{a b} = \state{a} \state{b},
\]
\[
\state{a b c}
= \state{a} \state{b} \state{c} + \left( \state{a c} - \state{a} \state{c} \right) \psi[b].
\]
\end{Example}

\begin{Example}
\label{Example:Conditional-freeness}
The following are important particular cases of conditional freeness.
\begin{enumerate}
\item
If $\phi = \psi$, so that $(\mc{A}, \phi)$ is an algebra with a single state, conditional freeness with respect to $(\phi, \phi)$ is the same as free independence with respect to $\phi$. Moreover, $R^{\phi, \phi} = R^\phi$.
\item
If $\mc{A}$ is a non-unital algebra, define a state $\delta_0$ on its unitization $\mf{C} 1 \oplus \mc{A}$ by $\delta_0[1] = 1$, $\delta_0[\mc{A}] = 0$. Then conditional freeness of subalgebras $(\mf{C} 1 \oplus \mc{A}_1), \ldots, (\mf{C} 1 \oplus \mc{A}_d)$ with respect to $(\phi, \delta_0)$ is the same as Boolean independence of subalgebras $\mc{A}_1, \ldots, \mc{A}_d$ with respect to $\phi$. Moreover, $R^{\phi, \delta_0} = \eta^\phi$.
\item
Specializing the preceding example, if $\mc{A} = \mf{C} \langle \mb{x} \rangle$, it is a unitization of the algebra of polynomials without constant term, and $\delta_0[P]$ is the constant term of a polynomial, so that we denote, even for non-commuting polynomials,
\[
\delta_0[P] = P(0).
\]
\item
More generally, for any $\mb{a} \in \mf{R}^d$, we can define a state on $\mf{C} \langle \mb{x} \rangle$ by
\[
\delta_{\mb{a}}[P] = P(a_1, a_2, \ldots, a_d).
\]
These are exactly the multiplicative $\ast$-linear functionals on $\mf{C} \langle \mb{x} \rangle$. Note that $\delta_{\mb{a}}$ is the free product of $\delta_{a_i}$, so it is indeed a state, and here and in all the preceding examples, the two parts of Definition~\ref{Defn:c-free} coincide. This particular case of the c-free theory is related to the objects in \cite{Krystek-Yoshida-t}, and also to Fermi independence (see Lemma~\ref{Lemma:B_t-semigroup}).
\end{enumerate}
\end{Example}

\noindent
See the references for other particular cases and generalizations of conditional freeness; the appearance of the free Meixner laws (Section~\ref{Subsec:Meixner}) in related contexts has been observed even more widely. Note also that the only other natural product, the monotone product of Muraki, can also to some degree be handled in the c-free language \cite{Franz-Multiplicative-monotone}: the monotone product of $\phi_1$ with $\phi_2$ is their conditionally free product for the second state $\delta_0 \ast \phi_2$. However, while the monotone product is associative, the triple product is apparently not a conditionally free product. See also Remark~\ref{Remark:Monotone} and Lemma~\ref{Lemma:Monotone-Phi}.

\subsection{Convolutions}
\label{Subsec:Convolutions}
If $\phi, \psi$ are two unital linear functionals on $\mf{C} \langle \mb{x} \rangle$, then $\phi \boxplus \psi$ is their free convolution, that is a unital linear functional on $\mf{C} \langle \mb{x} \rangle$ determined by
\begin{equation}
\label{Free-convolution}
\CumFun{\phi}{\mb{z}} + \CumFun{\psi}{\mb{z}} = \CumFun{\phi \boxplus \psi}{\mb{z}}.
\end{equation}
Similarly, $\phi \uplus \psi$, their Boolean convolution, is a unital linear functional on $\mf{C} \langle \mb{x} \rangle$ determined by
\[
\eta^\phi(\mb{z}) + \eta^\psi(\mb{z}) = \eta^{\phi \uplus \psi}(\mb{z}).
\]
See Lecture~12 of \cite{Nica-Speicher-book} for the relation between free convolution and free independence; the relation in the Boolean case is similar.

\br
The functionals $\rho^{\boxplus t}$ form a free convolution semigroup if $\rho^{\boxplus t} \boxplus \rho^{\boxplus s} = \rho^{\boxplus (t+s)}$. If $\rho^{\boxplus t}$ are \emph{states} for all $t \geq 0$, we say that $\rho$ is freely infinitely divisible, and will denote the corresponding semigroup by $\set{\rho_t}$. The functionals $\phi^{\uplus t}$ form a Boolean convolution semigroup if $\phi^{\uplus t} \uplus \phi^{\uplus s} = \phi^{\uplus (t+s)}$. Any state is infinitely divisible in the Boolean sense.

\subsection{Free Meixner distributions and states}
\label{Subsec:Meixner}
The semicircular distribution with mean $\alpha$ and variance $\beta$ is
\[
d\SC(\alpha, \beta)(x) = \frac{1}{2 \pi \beta} \sqrt{(4 \beta - (x - \alpha)^2} \chf{[-2 \sqrt{\beta}, 2 \sqrt{\beta}]}(x) \,dx
\]
For every $\alpha, \beta$, $\SC(\alpha t, \beta t)$ form a free convolution semigroup with respect to $t$.

\br
For $b \in \mf{R}$, $1 + c \geq 0$, the free Meixner distributions, normalized to have mean zero and variance one, are
\[
d\mu_{b,c}(x) = \frac{1}{2 \pi} \frac{\sqrt{4 (1 + c) - (x - b)^2}}{1 + b x + c x^2} \,dx + \text{ zero, one, or two atoms}.
\]
They are characterized by their Jacobi parameter sequences having the special form
\[
(0, b, b, b, \ldots), (1, 1+c, 1+c, 1+c, \ldots),
\]
or by the special form of the generating function of their orthogonal polynomials. In particular, $\mu_{0,0} = \SC(0,1)$ is the standard semicircular distribution, $\mu_{b,0}$ are the centered free Poisson distributions, and $\mu_{b,-1}$ are the normalized Bernoulli distributions. Moreover,
\[
\mu_{b,c} = \SC(b,1+c)^{\uplus 1/(1+c)} \uplus \delta_{-b}.
\]
Un-normalized free Meixner distributions, of mean $\alpha$ and variance $t$, are $\mu_{b,c}^{\boxplus t} \boxplus \delta_{\alpha}$ or $\mu_{\beta, \gamma}^{\uplus t} \uplus \delta_\alpha$, see Lemma~\ref{Lemma:Meixner}.

\br
More generally, free Meixner states are states on $\mf{C} \langle \mb{x} \rangle$, characterized by a number of equivalent conditions (see \cite{AnsFree-Meixner}), among them the equations
\begin{equation}
\label{free-quadratic-PDE}
D_i D_j \CumFun{\phi}{\mb{z}} = \delta_{ij} + \sum_{k=1}^d B_{ij}^k D_k \CumFun{\phi}{\mb{z}} + C_{ij} D_i \CumFun{\phi}{\mb{z}} D_j \CumFun{\phi}{\mb{z}}.
\end{equation}
for certain $\set{B_{ij}^k, C_{ij}}$. In \cite{AnsBoolean}, these equations were shown to be equivalent to
\begin{equation}
\label{Boolean-quadratic-PDE}
D_i D_j \eta^\phi(\mb{z}) = \delta_{ij} + \sum_{k=1}^d B_{ij}^k D_k \eta^\phi(\mb{z}) + (1 + C_{ij}) D_i \eta^\phi(\mb{z}) D_j \eta^\phi(\mb{z}).
\end{equation}

\subsection{Orthogonality of the c-free Appell polynomials}
The c-free version of the Appell polynomials were investigated in \cite{AnsAppell3}. A result which motivated the investigation of the free Meixner distributions below is the following lemma.

\begin{Lemma}[Lemma~5, Theorem~6 of \cite{AnsAppell3}]
The c-free Appell polynomials in $(\mc{A}, \phi, \psi)$ are orthogonal if and only if $\psi$ is a free product of semicircular distributions $\SC(b_i, 1 + c_i)$, and $\phi = \State{\psi}$. In this case, $\CumFun{\phi, \psi}{\mb{z}} = \sum_{i=1}^d z_i^2$, and $\phi$ is a free Meixner state.
\end{Lemma}

\noindent
Here $\Phi$ is a map defined by Belinschi and Nica in \cite{Belinschi-Nica-Free-BM} via
\[
\eta^{\State{\psi}}(\mb{w}) = \sum_{i=1}^d w_i (1 + M^\psi(\mb{w})) w_i,
\]
see \cite{AnsAppell3} for other descriptions of it.

\section{Free Meixner distributions and conditional freeness}
\label{Section:Meixner}

\begin{Lemma}
\label{Lemma:boxplus-uplus}
For two functionals $\phi, \psi$, $\CumFun{\phi, \psi}{\mb{z}} = \sum_{i=1}^d a_i z_i + \beta \CumFun{\psi}{\mb{z}}$ if and only if
\[
\phi = \delta_{\mb{a}} \uplus \psi^{\uplus \beta}.
\]
Equivalently, if $\rho$ is a (not necessarily positive) functional with $\CumFun{\rho}{\mb{z}} = \CumFun{\phi, \psi}{\mb{z}}$ and $\CumFun{\psi}{\mb{z}} = \sum_{i=1}^d a_i z_i + t \CumFun{\rho}{\mb{z}}$, then
\[
\psi = \delta_{\mb{a}} \boxplus \rho^{\boxplus t} = \delta_{\mb{a}} \uplus \phi^{\uplus t}.
\]
\end{Lemma}

\begin{proof}
If $\CumFun{\phi, \psi}{\mb{z}} = \sum_{i=1}^d a_i z_i + \beta \CumFun{\psi}{\mb{z}}$, then
\[
\begin{split}
\eta^\phi(\mb{w})
& = (1 + M^\psi(\mb{w}))^{-1} \CumFun{\phi, \psi}{(1 + M^\psi(\mb{w})) \mb{w}} \\
& = \sum_{i=1}^d a_i w_i + \beta (1 + M^\psi(\mb{w}))^{-1} \CumFun{\psi}{(1 + M^\psi(\mb{w})) \mb{w}} \\
& = \sum_{i=1}^d a_i w_i + \beta (1 + M^\psi(\mb{w}))^{-1} M^\psi(\mb{w})
= \sum_{i=1}^d a_i w_i + \beta \eta^\psi(\mb{w})
= \eta^{\delta_{\mb{a}}} + \eta^{\psi^{\uplus \beta}},
\end{split}
\]
so
\[
\phi = \delta_{\mb{a}} \uplus \psi^{\uplus \beta}.
\]
Reversing the argument gives the reverse implication. The proof of the second statement is similar.
\end{proof}

\begin{Defn-Remark}
For $\rho, \psi$ linear functionals, define the map
\[
(\rho, \psi) \mapsto \State{\rho, \psi}
\]
via
\[
\eta^{\State{\rho, \psi}}(\mb{w}) = (1 + M^\psi(\mb{w}))^{-1} \CumFun{\rho}{(1 + M^\psi(\mb{w})) \mb{w}}.
\]
In other words, for
\[
\phi = \State{\rho, \psi}
\]
we have
\[
\CumFun{\rho}{\mb{z}} = \CumFun{\phi, \psi}{\mb{z}}.
\]
Note that $\CumFun{\SC(0,1)}{\mb{z}} = \sum_{i=1}^d z_i^2$, so that
\[
\eta^{\State{\SC(0,1), \psi}}(\mb{w}) = (1 + M^\psi(\mb{w}))^{-1} \sum_{i=1}^d (1 + M^\psi(\mb{w})) w_i (1 + M^\psi(\mb{w})) w_i
\]
and
\[
\State{\SC(0,1), \psi} = \State{\psi}.
\]
On the other hand, from $R^{\phi, \phi} = R^\phi$ and $R^{\phi, \delta_0} = \eta^\phi$ it follows that for any $\rho$, $\State{\rho, \rho} = \rho$ and $\mf{B}[\State{\rho, \delta_0}] = \rho$, where the map $\mf{B}$ is defined in the next remark. See also Remark~\ref{Remark:Lenczewski}.

\br
We will discuss in the next section under what conditions $\State{\rho, \psi}$ is a state.
\end{Defn-Remark}

\begin{Defn-Remark}
For $\mb{a} \in \mf{R}^d$, define the transformation $\mf{B}_{\mb{a}, t}$ by
\[
\mf{B}_{\mb{a}, t}[\rho] = \Bigl((\rho^{\boxplus (1 + t)} \boxplus \delta_{\mb{a}}) \uplus \delta_{-\mb{a}} \Bigr)^{\uplus 1/(1+t)}.
\]
$\mf{B}_{\mb{a}, t}$ maps functionals to functionals; for $t \geq 0$, it maps states to states. For $t > -1$, $\mf{B}_{\mb{a}, t}$ maps freely infinitely divisible states to states; for smaller $t$, the domains of $\mf{B}_t$ as a map from states to states get progressively smaller. For $\mb{a} = 0$,
\[
\mf{B}_{0, t}[\rho] = \mf{B}_t[\rho] = (\rho^{\boxplus (1+t)})^{\uplus (1/(1+t))}
\]
is the Belinschi-Nica transformation \cite{Belinschi-Nica-B_t,Belinschi-Nica-Free-BM}. In particular, $\mf{B}_{0,1} = \mf{B}$, the Boolean-to-free version of the Bercovici-Pata bijection. On the other hand, we will show below that for each $\rho$,
\[
\mf{B}_{(\rho[x_1], \ldots, \rho[x_d]), 0}[\rho]
\]
is the image of $\rho$ under the Boolean-to-Fermi version of the Bercovici-Pata bijection, in the sense of \cite{Oravecz-Fermi}.

\br
The preceding lemma states that if $\CumFun{\rho}{\mb{z}} = \CumFun{\phi, \psi}{\mb{z}}$ and $\CumFun{\psi}{\mb{z}} = \sum_{i=1}^d a_i z_i + (1 + t) \CumFun{\rho}{\mb{z}}$, then $\phi = \mf{B}_{\mb{a}, t}[\rho]$. In other words,
\begin{equation}
\label{Phi-rho-1+t}
\phi = \mf{B}_{\mb{a}, t}[\rho] = \State{\rho, \rho^{\boxplus (1+t)} \boxplus \delta_{\mb{a}}} = \State{\rho, \rho \boxplus \rho^{\boxplus t} \boxplus \delta_{\mb{a}}}.
\end{equation}
We will generalize this result in the next section. Note also that in this case, $\psi = U_{\mb{t}}(\phi)$ in the sense of \cite{Krystek-Yoshida-t}.
\end{Defn-Remark}

\noindent
The following lemma extends Proposition~8 from \cite{AnsBoolean} and Example~4.5 and Remark~4.6 from \cite{Belinschi-Nica-B_t}.

\begin{Lemma}
\label{Lemma:Meixner}
\br
\begin{enumerate}
\item
For $\mu_{b,c}$ a free Meixner distribution,
\[
\mf{B}_{\alpha, t}[\mu_{b,c}] = \mu_{b + \alpha, c + t}
\]
is also a free Meixner distribution.
\item
More generally, for $\phi_{\set{T_i}, C}$ a free Meixner state,
\[
\mf{B}_{\mb{a}, t}[\phi_{\set{T_i}, C}] = \phi_{\set{a_i I + T_i}, tI + C}
\]
is also a free Meixner state.
\item
Any free Meixner distribution can be obtained from the Bernoulli distribution
\[
\mu_{0,-1} = \frac{1}{2} \delta_{-1} + \frac{1}{2} \delta_1
\]
by the application of the appropriate $\mf{B}_{\alpha, t}$.
\item
For $\rho = \mu_{b,c}$ a free Meixner distribution,
\[
\State{\mu_{b,c}, \mu_{b,c}^{\boxplus (1+t)} \boxplus \delta_\alpha} = \mf{B}_{\alpha, t}[\mu_{b,c}] = \mu_{b + \alpha, c + t}.
\]
A partial converse to this statement is given in Proposition~\ref{Prop:Meixner-characterization}.
\end{enumerate}
\end{Lemma}

\begin{proof}
For part (a), it suffices to show that
\[
\delta_{\beta - b} \boxplus \mu_{b, c}^{\boxplus (1 + \gamma - c)} = \delta_{\beta - b} \uplus \mu_{\beta, \gamma}^{\uplus (1 + \gamma - c)}.
\]
$\mu_{b,c}$ has Jacobi parameter sequences
\[
\set{(0, b, b, b, \ldots,), (1, 1+c, 1+c, 1+c, \ldots)},
\]
As a particular case of the results in Section 3.1 of \cite{AnsFree-Meixner}, it follows that $\delta_{\beta - b} \boxplus \mu_{b, c}^{\boxplus (1 + \gamma - c)}$ has Jacobi parameter sequences
\[
\set{(\beta - b, \beta, \beta, \beta, \ldots,), (1 + \gamma - c, 1 + \gamma, 1 + \gamma, 1 + \gamma, \ldots)}.
\]
On the other hand, as observed in \cite{Boz-Wys} or the Appendix of \cite{AnsBoolean}, for any measure $\mu$ with Jacobi parameters
\[
\set{(\beta_0, \beta_1, \beta_2, \ldots), (\gamma_1, \gamma_2, \gamma_3, \ldots)},
\]
the measure $\delta_\alpha \uplus \mu^{\uplus t}$ has Jacobi parameters
\[
\set{(\alpha + t \beta_0, \beta_1, \beta_2, \ldots), (t \gamma_1, \gamma_2, \gamma_3, \ldots)}.
\]
The result follows. The proof for part (b) is similar. For part (c),
\[
\mu_{b,c} = \mf{B}_{b,1+c}[\mu_{0,-1}].
\]
Part (d) follows by combining part (a) with equation~\eqref{Phi-rho-1+t}.
\end{proof}

\begin{Example}
In the c-free central limit theorem (Theorem~4.3 of \cite{BLS96}), for the limiting distribution $\CumFun{\rho}{z} = \CumFun{\phi, \psi}{z} = z^2$, so that $\rho = \mu_{0,0}$ is the semicircular distribution and
\[
\phi = \State{\psi}
\]
If moreover $\CumFun{\psi}{z} = t z^2$, that is $\psi = \rho^{\boxplus t}$, then
\[
\phi = \State{\mu_{0,0}^{\boxplus t}} = \mu_{0, -1 + t}
\]
so that $\phi$ can be any symmetric free Meixner distribution.

\br
In the c-free Poisson limit theorem (centered version of Theorem~4.4 of \cite{BLS96}), for the limiting distribution $\CumFun{\rho}{z} = \CumFun{\phi, \psi}{z} = \frac{z^2}{1 - b z}$, so $\rho = \mu_{b,0}$ is the centered free Poisson distribution. Using Lemma~\ref{Lemma:Monotone-Phi}, it follows that
\[
\phi = \State{\mu_{b,0}, \psi} = \State{\delta_b \rhd \psi} = \State{\delta_b \uplus \psi}.
\]
If moreover $\CumFun{\psi}{z} = t \frac{z^2}{1 - a z}$ for $a = b$, then $\psi = \rho^{\boxplus t}$ and
\[
\phi = \State{\rho, \psi} = \State{\mu_{b,0}, \mu_{b,0}^{\boxplus t}} = \mu_{b, -1 + t},
\]
so that $\phi$ can be any free Meixner distribution. If $a \neq b$, then
\[
\phi = \State{\delta_b \uplus \mu_{a,0}^{\boxplus t}} = \State{\delta_b \uplus \mf{B}_{t-1}[\mu_{a,0}]^{\uplus t}} = \State{\delta_b \uplus \mu_{a,-1+t}^{\uplus t}} = \State{\delta_{b-a} \uplus \SC(a,t)},
\]
which is a distribution whose Jacobi parameter sequences are constant after step three, cf.\ Theorems 11 and 12 of \cite{Kry-Woj-Associative}.

\br
Similar arguments explain the appearance of the free Meixner distributions in various contexts which can be derived from the c-free formalism.

\br
Finally, if $\phi = \mu_{\beta, \gamma}$ and $\rho = \mu_{b, c}$ are both free Meixner distributions, as long as
\[
(c \leq 0, \gamma \geq c) \text{ or } 1 + \gamma \geq c > 0,
\]
$\CumFun{\phi, \psi}{z} = \CumFun{\rho}{z}$ for $\psi = \mu_{b, c}^{\boxplus (1 + \gamma - c)} \boxplus \delta_{\beta - b}$. Moreover, for arbitrary free Meixner distributions, $\mu_{\beta, \gamma} = \mf{B}_{\beta - b, \gamma - c}[\mu_{b, c}]$. In this case $1 + \gamma - c$ may be negative, however $\mu_{b, c}$ is in the range of $\mf{B}_{\alpha, 1 + c}$ (its $\boxplus$-divisibility indicator in the sense of equation~1.7 of \cite{Belinschi-Nica-B_t} is $1+c$), and so in the domain of $\mf{B}_{\alpha, t}$ for $t \geq - (1 + c)$.

\br
Similar results hold in the multivariate case. See also Remark~\ref{Remark:Lenczewski}.
\end{Example}

\begin{Prop}
If $\phi$ is a free Meixner state, then
\[
\begin{split}
D_i D_j \CumFun{\phi, \psi}{\mb{z}}
& = \delta_{ij} + \sum_k B_{ij}^k D_k \CumFun{\phi, \psi}{\mb{z}} + (1 + C_{ij}) D_i \CumFun{\phi, \psi}{\mb{z}} D_j \CumFun{\phi, \psi}{\mb{z}} \\
&\quad - D_i \CumFun{\psi}{\mb{z}} D_j \CumFun{\phi, \psi}{\mb{z}}.
\end{split}
\]
Note that this reduces to equation~\eqref{free-quadratic-PDE} for $\psi = \phi$, and to equation~\eqref{Boolean-quadratic-PDE} for $\psi = \delta_0$.
\end{Prop}

\begin{proof}
From relation~\eqref{eta-M-R}, it follows that
\[
D_j \eta^{\phi}(\mb{w})
= (D_j R^{\phi, \psi}) \left((1 + M^\psi(\mb{w})) \mb{w}\right)
\]
and
\[
\begin{split}
(D_i D_j R^{\phi, \psi}) \left((1 + M^\psi(\mb{w})) \mb{w}\right)
& = D_i \left[ (1 + M^\psi(\mb{w}))^{-1} (D_j R^{\phi, \psi}) \left((1 + M^\psi(\mb{w})) \mb{w}\right) \right] \\
& = D_i \left[ (1 + M^\psi(\mb{w}))^{-1} D_j \eta^\phi(\mb{w}) \right] \\
& = D_i (1 + M^\psi(\mb{w}))^{-1} D_j \eta^\phi(\mb{w}) + D_i D_j \eta^\phi(\mb{w}) \\
& = - D_i \eta^\psi(\mb{w}) D_j \eta^\phi(\mb{w}) + D_i D_j \eta^\phi(\mb{w}).
\end{split}
\]
From equation~\eqref{Boolean-quadratic-PDE} for $\eta^\phi$ we get
\[
\begin{split}
D_i D_j \CumFun{\phi, \psi}{\mb{z}}
& = \delta_{ij} + \sum_k B_{ij}^k D_k \CumFun{\phi, \psi}{\mb{z}} + (1 + C_{ij}) D_i \CumFun{\phi, \psi}{\mb{z}} D_j \CumFun{\phi, \psi}{\mb{z}} \\
& \quad - D_i \eta^{\psi}(\mb{w}) D_j \CumFun{\phi, \psi}{\mb{z}}
\end{split}
\]
Since
\begin{equation}
\label{D-eta-D-R}
\begin{split}
D_i \eta^\psi(\mb{w})
& = D_i \left[(1 + M^\psi(\mb{w}))^{-1} M^\psi(\mb{w}) \right] \\
& = D_i \left[(1 + M^\psi(\mb{w}))^{-1} \CumFun{\psi}{(1 + M^\psi(\mb{w})) \mb{w}} \right]
= (D_i R^\psi)  \left( (1 + M^\psi(\mb{w})) \mb{w} \right)
\end{split}
\end{equation}
we finally get
\[
\begin{split}
D_i D_j \CumFun{\phi, \psi}{\mb{z}}
& = \delta_{ij} + \sum_k B_{ij}^k D_k \CumFun{\phi, \psi}{\mb{z}} + (1 + C_{ij}) D_i \CumFun{\phi, \psi}{\mb{z}} D_j \CumFun{\phi, \psi}{\mb{z}} \\
&\quad - D_i \CumFun{\psi}{\mb{z}} D_j \CumFun{\phi, \psi}{\mb{z}}. \qedhere
\end{split}
\]
\end{proof}

\begin{Remark}[Laha-Lukacs relations]
In \cite{Laha-Lukacs}, Laha and Lukacs proved that the Meixner distributions are characterized by a certain property involving linear conditional expectations and quadratic conditional variances. In \cite{Boz-Bryc}, Bo{\.z}ejko and Bryc proved that an identical characterization holds, in the free setting, for the free Meixner distributions. In \cite{Boz-Bryc-Two-states}, they further characterized all distributions having the Laha-Lukacs property in the two-state setting. In our notation, and with some extra assumptions, their result states that in this case,
\[
\CumFun{\phi, \psi}{z} = \frac{z^2}{1 - b z - c \CumFun{\psi}{z}},
\]
or equivalently
\[
D^2 \CumFun{\phi, \psi}{z} = 1 + b D \CumFun{\phi, \psi}{z} + c D \CumFun{\psi}{z} D \CumFun{\phi, \psi}{z}.
\]
Thus for $b=c=0$, one gets $\CumFun{\phi, \psi}{z} = z^2$ and $\phi = \State{\psi}$, while for $c = 0$, one gets $\CumFun{\phi, \psi}{z} = \frac{z^2}{1 - b z}$ and the appropriate analog of the Poisson law. On the other hand, for general $b,c$, $\CumFun{\phi, \psi}{z}$ does not itself satisfy a quadratic equation for general $\psi$, so the general two-state Laha-Lukacs distributions are not directly related to the free Meixner states. This was already observed in Section 4.4 of \cite{AnsBoolean}, where the Boolean Laha-Lukacs distributions (a particular case of the result in \cite{Boz-Bryc-Two-states}) were found to be the Bernoulli distributions, which include the Boolean analogs of the normal and Poisson laws, but do not include all the free Meixner distributions. On the other hand, if $\CumFun{\phi, \psi}{z} = \CumFun{\psi}{z}$, in which case also $\phi = \psi$, one recovers a quadratic equation, all free Meixner distributions, and free probability.
\end{Remark}

\section{The transformations $\mf{B}_{\mb{a}, t}$ and $\Phi[\rho, \psi]$}
\label{Section:Transformations}

\subsection{Properties of $\mf{B}_{\mb{a}, t}$}

\begin{Lemma}
\label{Lemma:B_t-semigroup}
The maps $\mf{B}_{\mb{a}, t}$ form a semigroup:
\[
\mf{B}_{\mb{a}, t} \circ \mf{B}_{\mb{b}, s} = \mf{B}_{\mb{a} + \mb{b}, t + s}.
\]
In particular, $\mf{B}_t$ and $\mf{B}_{\mb{a}, 0}$ commute, and
\[
\mf{B}_{\mb{a}, t}
= \mf{B}_t \circ \mf{B}_{\mb{a}, 0}.
\]
Moreover, for each $\rho$,
\[
\mf{B}_{(\rho[x_1], \ldots, \rho[x_d]), 0}[\rho]
\]
is the image of $\rho$ under the Boolean-to-Fermi version of the Bercovici-Pata bijection, in the sense of \cite{Oravecz-Fermi}.
\end{Lemma}

\begin{proof}
The representation
\[
\eta^{\mf{B}_{\mb{a}, t}[\rho]}[x_i] = \eta^\rho[x_i] = \rho[x_i]
\]
and for $n > 1$, $\abs{\vec{u}} = n$,
\[
\eta^{\mf{B}_{\mb{a}, t}[\rho]}[x_{u(1)}, x_{u(2)}, \ldots, x_{u(n)}] = \sum_{\pi \in \NC'(n)} \sum_{S \subset \Sing(\pi)} \Bigl( \prod_{i \in S} a_{u(i)} \Bigr) t^{\abs{S^c} - 1} \prod_{B \in S^c} \eta^\rho[x_{u(i)}: i \in B]
\]
is obtained by the same methods as in \cite{Belinschi-Nica-Eta}. In particular,
\[
\eta^{\mf{B}_{0, t}[\rho]}[x_{u(1)}, x_{u(2)}, \ldots, x_{u(n)}] = \sum_{\pi \in \NC'(n)} t^{\abs{\pi} - 1} \prod_{B \in \pi} \eta^\rho[x_{u(i)}: i \in B]
\]
and
\begin{equation}
\label{Formula-B_a}
\eta^{\mf{B}_{\mb{a}, 0}[\rho]}[x_{u(1)}, x_{u(2)}, \ldots, x_{u(n)}] = \sum_{\set{1,n} \subset \Lambda \subset \set{1, 2, \ldots, n}} \Bigl( \prod_{i \not \in \Lambda} a_{u(i)} \Bigr) \eta^\rho[x_{u(i)}: i \in \Lambda]
\end{equation}
which are easily seen to commute, and which proves the semigroup property.

\br
Moreover, denoting $\overrightarrow{\rho[x]} = (\rho[x_1], \rho[x_2], \ldots, \rho[x_d])$,
\[
\begin{split}
\eta^{\mf{B}_{\overrightarrow{\rho[x]}, 0}[\rho]}[x_{u(1)}, x_{u(2)}, \ldots, x_{u(n)}]
& = \sum_{\set{1,n} \subset \Lambda \subset \set{1, 2, \ldots, n}} \Bigl( \prod_{i \not \in \Lambda} \rho[x_{u(i)}] \Bigr) \eta^{\rho} [x_{u(i)}: i \in \Lambda] \\
& = \sum_{\substack{\pi \in \NC'(n) \\ \Inner(\pi) \subset \Sing(\pi)}} \Bigl( \prod_{\set{i} \in \Inner(\pi)} \rho[x_{u(i)}] \Bigr)  \eta^{\rho} [x_{u(i)}: i \in B^o(\pi)].
\end{split}
\]
Using Definition~2.2 from \cite{Oravecz-Fermi}, it now easily follows that the Boolean cumulants of $\rho$ are the Fermi cumulants of $\mf{B}_{\overrightarrow{\rho[x]}, 0}[\rho]$. In other words, for each $\rho$, $\mf{B}_{\overrightarrow{\rho[x]}, 0}[\rho]$ is the image of $\rho$ under the Boolean-to-Fermi version of the Bercovici-Pata bijection.
\end{proof}

\subsection{Properties of $\State{\cdot, \cdot}$}

\begin{Thm}
\label{Thm:c-free-representation}
On $\mf{C} \langle x_1, x_2, \ldots, x_d \rangle$, let $\psi$ be a positive definite functional and $\mu$ a conditionally positive definite functional. Define a functional $\eta$ via
\[
M^\eta(\mb{w}) = (1 + M^\psi(\mb{w}))^{-1} M^\mu((1 + M^\psi(\mb{w})) \mb{w}) .
\]
Then $\eta$ is conditionally positive definite.
\end{Thm}

\noindent
The proof is delayed until Section~\ref{Subsec:Operator-models}.

\begin{Cor}
\label{Cor:cpd-one-variable}
Let $f, g$ be power series whose coefficient sequences are positive definite. Then the coefficient sequence of
\[
f(z g(z)) g(z)
\]
is also positive definite.
\end{Cor}

\begin{proof}
In a single variable, a sequence $\set{m_0, m_1, m_2, \ldots}$ is conditionally positive definite if and only if the sequence $\set{m_2, m_3, m_4, \ldots}$ is positive definite. In particular, the coefficient sequence of $z^2 f(z)$ is conditionally positive definite. So by Theorem~\ref{Thm:c-free-representation}, the coefficient sequence of
\[
g(z)^{-1} (z g(z))^2 f(z g(z)) = z^2 f(z g(z) g(z)
\]
is conditionally positive definite, and therefore the coefficient sequence of $f(z g(z)) g(z)$ is positive definite.
\end{proof}

\noindent
The following corollary was proved in Lemma 6.1 of \cite{Krystek-Conditional}, in the one-variable compactly supported case, by complex-analytic methods.

\begin{Cor}
\label{Cor:Phi-positive-definite}
Let $\psi$ be a state. For any conditionally positive functional $\mu$, there exists a state $\phi$ such that $R^{\phi, \psi} = M^\mu$. Equivalently, for any freely infinitely divisible state $\rho$, there exists a state $\phi$ such that $R^{\phi, \psi} = R^\rho$.
\end{Cor}

\begin{proof}
The conditionally positive functional $\eta$ obtained in Theorem~\ref{Thm:c-free-representation} is necessarily a Boolean cumulant functional of a state $\phi$. The second statement follows from the fact that the free cumulant functional of a freely infinitely divisible state is conditionally positive definite.
\end{proof}

\begin{Remark}[Schoenberg correspondence]
\label{Remark:Schoenberg}
Another interpretation of this result is in terms of the Schoenberg correspondence, see Corollary~3.6 of \cite{Franz-Unification} or \cite{SchurCondPos}. Let $\psi$ be any freely infinitely divisible state. A functional $\mu = R^\rho$ is conditionally positive if and only if it a generator of a c-free convolution semigroup of pairs of states $(\phi(t), \psi(t))$, in the sense that
\[
\mu = \frac{d}{dt}\Bigl|_{t=0}\Bigr. \phi(t),
\]
with $\psi(1) = \psi$. Indeed, for $\psi(t) = \psi^{\boxplus t}$, and $\phi(t)$ chosen so that
\[
R^{\phi(t), \psi^{\boxplus t}} = t M^\mu = t R^\rho = R^{\rho^{\boxplus t}},
\]
in other words for $\phi(t) = \State{\rho^{\boxplus t}, \psi^{\boxplus t}}$, the two properties above hold. The converse is standard. Note that we recover the Boolean version of the correspondence for $\psi = \delta_0$, and the free version for $\psi = \rho$.
\end{Remark}

\begin{Defn-Remark}
\label{DEfn-Remark-Monotone}
For two functionals $\tau, \psi$, the monotone convolution $\tau \rhd \psi$ of Muraki \cite{Franz-Muraki-Markov-monotone} is determined by
\begin{equation}
\label{Monotone-convolution}
(1 + M^{\tau \rhd \psi}(\mb{w})) = \Bigl(1 + M^\tau \bigl((1 + M^\psi(\mb{w})) \mb{w} \bigr) \Bigr) (1 + M^\psi(\mb{w})).
\end{equation}
So far, this operation has apparently been considered only in one variable, in which case this equation is equivalent to the condition
\[
F_{\tau \rhd \psi}(w) = F_{\tau}(F_\psi(w))
\]
on the reciprocal Cauchy transforms of the corresponding measures. It follows from Remark~\ref{Remark:Monotone} that the monotone convolution of states is a state.
\end{Defn-Remark}

\begin{Lemma}
\label{Lemma:Monotone-Phi}
Let $\psi$ be a state.
\begin{enumerate}
\item
Suppose that $R^\rho$ has the special form
\begin{equation}
\label{free-Phi}
\CumFun{\rho}{\mb{z}} = \sum_{i=1}^d z_i (1 + M^\tau(\mb{z})) z_i,
\end{equation}
which we will denote by $\rho = \Phi_{\text{free}}[\tau]$, where $\tau$ is a state. Note that $\Phi_{\text{free}}[\tau] = \mf{B}[\State{\tau}]$. Then
\begin{equation}
\label{Phi-monotone}
\State{\tau \rhd \psi} = \State{\Phi_{\text{free}}[\tau], \psi}.
\end{equation}
\item $\delta_{\mb{a}} \rhd \psi = \delta_{\mb{a}} \uplus \psi$.
\end{enumerate}
\end{Lemma}

\begin{proof}
Since $\tau$ is a state, by Lemmas 12 and 13 of \cite{AnsMulti-Sheffer} $R^\rho$ is conditionally positive definite and so $\rho$ is freely infinitely divisible. Then
\[
(1 + M^\psi(\mb{w}))^{-1} \CumFun{\rho}{(1 + M^\psi(\mb{w})) \mb{w}}
= \sum_{i=1}^d w_i \Bigl(1 + M^\tau \bigl((1 + M^\psi(\mb{w})) \mb{w} \bigr) \Bigr) (1 + M^\psi(\mb{w})) w_i
\]
so that for $\phi = \State{\rho, \psi}$,
\[
\eta^\phi(\mb{w}) = \sum_{i=1}^d w_i (1 + M^{\tau \rhd \psi}(\mb{w})) w_i.
\]
Equation~\eqref{Phi-monotone} follows.

\br
For part (b), by definition
\[
\begin{split}
(1 + M^{\delta_{\mb{a}} \rhd \psi}(\mb{w}))
& = \Bigl(1 + M^{\delta_{\mb{a}}} \bigl((1 + M^\psi(\mb{w})) \mb{w} \bigr) \Bigr) (1 + M^\psi(\mb{w})) \\
& = \Bigl( 1 - \sum_{i=1}^d a_i (1 + M^\psi(\mb{w})) w_i \Bigr)^{-1} (1 + M^\psi(\mb{w})),
\end{split}
\]
so
\[
\begin{split}
\eta^{\delta_{\mb{a}} \rhd \psi}(\mb{w})
& = 1 - \Bigl(1 + M^{\delta_{\mb{a}} \rhd \psi}(\mb{w}) \Bigr)^{-1} \\
& = 1 - (1 + M^\psi(\mb{w}))^{-1} \Bigl( 1 - \sum_{i=1}^d a_i (1 + M^\psi(\mb{w})) w_i \Bigr) \\
& = 1 - (1 + M^\psi(\mb{w}))^{-1} + \sum_{i=1}^d a_i w_i
= \eta^{\delta_{\mb{a}} \uplus \psi}(\mb{w}). \qedhere
\end{split}
\]
\end{proof}

\begin{Thm}
\label{Thm:Phi-properties}
\br
\begin{enumerate}
\item
For $\phi = \State{\rho, \psi}$, any two of the functionals $(\phi, \rho, \psi)$ uniquely determine the third.
\item
$\State{\cdot, \cdot}$ is a well-defined map
\[
(\rho = \text{ freely infinitely divisible state, } \psi = \text{ state}) \mapsto (\phi = \text{ state}).
\]
From now on, unless stated otherwise, we will consider $\State{\cdot, \cdot}$ with these domain and range. Note that $\phi$ has the same mean and covariance as $\rho$.
\item
$\rho$ is the unique fixed point of the maps $\State{\rho, \cdot}$ and $\State{\cdot, \rho}$. In particular, a free product of standard semicircular distributions is the unique fixed point of $\Phi$.
\item
For fixed $\rho$ with mean zero and identity covariance, the image of the single-variable map $\State{\rho, \cdot}$ consists of all states with the same mean and covariance as $\rho$ if and only if $\rho = \mu_{0,0}$ is semicircular, in which case $\State{\mu_{0,0}, \psi} = \State{\psi}$, or more generally if $\rho = \mu_{b,0}$ is centered free Poisson, in which case
\[
\State{\mu_{b,0}, \psi} = \State{\psi \uplus \delta_b}.
\]
In several variables, $\State{\rho, \cdot}$ is never onto.
\item
For fixed $\psi$, the map $\State{\cdot, \psi}$ is onto if and only if $\psi = \delta_0$, in which case $\State{\rho, \delta_0} = \mf{B}^{-1}[\rho]$, or more generally if $\psi = \delta_{\mb{a}}$, in which case
\[
\State{\rho, \delta_{\mb{a}}} = \mf{B}_{\mb{a}, -1}[\rho] = \mf{B}_{\mb{a}, 0} \circ \mf{B}^{-1}[\rho].
\]
\end{enumerate}
\end{Thm}

\begin{proof}
For part (a), we note that $R^\rho = R^{\phi, \psi}$, $\phi = \State{\rho, \psi}$, and $\psi$ is determined by
\begin{equation}
\label{psi-determined}
\Bigl((1 + M^\psi(\mb{w})) w_1, \ldots, (1 + M^\psi(\mb{w})) w_d \Bigr) = (\mb{D} R^\rho)^{\langle -1 \rangle} \Bigl( \mb{D} \eta^\phi(\mb{w}) \Bigr),
\end{equation}
where $(\mb{D} R^\rho)^{\langle -1 \rangle}(\mb{z})$ is the inverse of the $d$-tuple of power series
\[
\Bigl( D_1 \CumFun{\rho}{z_1, \ldots, z_d}, \ldots, D_d \CumFun{\rho}{z_1, \ldots, z_d} \Bigr)
\]
with respect to composition.

\br
Part (b) is a re-formulation of Corollary~\ref{Cor:Phi-positive-definite}.

\br
Since $\CumFun{\psi, \psi}{\mb{z}} = \CumFun{\psi}{\mb{z}}$, $\State{\psi, \psi} = \psi$. So by part (a), $\State{\rho, \psi} = \psi = \State{\psi, \psi}$ if and only if $\psi = \rho$. Similarly, $\State{\rho, \psi} = \rho = \State{\rho, \rho}$ if and only if $\rho = \psi$.

\br
If the map $\State{\cdot, \cdot}$ is onto, in particular the Bernoulli distribution, with $\eta^\phi(\mb{w}) = \sum_{i=1}^d w_i^2$, is in the image. In this case, using equation~\eqref{D-eta-D-R},
\[
(D_i R^{\rho}) \left((1 + M^\psi(\mb{w})) \mb{w} \right) = D_i \eta^\phi(\mb{w}) = w_i,
\]
so that
\[
D_i \CumFun{\rho}{\mb{z}} = \Bigl(1 + \CumFun{\psi}{\mb{z}} \Bigr)^{-1} z_i
\]
and
\[
\CumFun{\rho}{\mb{z}} = \sum_{i=1}^d z_i \Bigl(1 + \CumFun{\psi}{\mb{z}} \Bigr)^{-1} z_i.
\]
Thus in the notation of the preceding Lemma, $\rho = \Phi_{\text{free}}[\tau]$, where
\begin{equation}
\label{M-R-inverse}
1 + M^\tau(\mb{z}) = \Bigl(1 + \CumFun{\psi}{\mb{z}} \Bigr)^{-1}
\end{equation}
Since $\rho$ is freely infinitely divisible, $R^\rho$ is conditionally positive definite, and therefore $\tau$ is positive definite. Indeed, for any $i$
\[
\tau[P(\mb{x})^\ast P(\mb{x})]
= \Cum{\rho}{x_i P(\mb{x})^\ast P(\mb{x}) x_i}
= \Cum{\rho}{(P(\mb{x}) x_i)^\ast (P(\mb{x}) x_i)} \geq 0.
\]
On the other hand, from equation \eqref{M-R-inverse}
\begin{multline*}
1 + \sum_i \tau[x_i] z_i + \sum_{i,j} \tau[x_i x_j] z_i z_j + \ldots \\
= 1 - \sum_i \Cum{\psi}{x_i} z_i - \sum_{i,j} \Cum{\psi}{x_i x_j} z_i z_j + \Bigl( \sum_i \Cum{\psi}{x_i} z_i \Bigr) \Bigl( \sum_j \Cum{\psi}{x_j} z_j \Bigr),
\end{multline*}
so that $\tau[x_i] = - \psi[x_i]$ and
\[
\Cum{\tau}{x_i x_j} = \tau[x_i x_j] - \tau[x_i] \tau[x_j] = - \Cum{\psi}{x_i x_j}.
\]
Thus the covariance matrices of $\psi, \tau$ differ by a sign. On the other hand, since $\tau, \psi$ are positive definite, so are their covariance matrices. It follows that these matrices are both zero. Therefore both $\tau$ and $\psi$ are multiplicative linear functionals, and so delta measures.

\br
If $\tau = \delta_{\mb{a}}$, then $1 + M^\tau(\mb{z}) = (1 - \sum_{i=1}^d a_i z_i)^{-1}$ and $\CumFun{\rho}{\mb{z}} = \sum_{j=1}^d z_j (1 - \sum_{i=1}^d a_i z_i)^{-1} z_j$, so that
\[
D_i D_j \CumFun{\rho}{\mb{z}} = \delta_{ij} + a_i D_j \CumFun{\rho}{\mb{z}}
\]
and $\rho$ is the free Meixner state $\phi_{\set{a_i I}, 0}$, the free product of centered free Poisson distributions. $\rho$ has the special form in the equation~\eqref{free-Phi}, so by the preceding Lemma,
\[
\State{\Phi_{\text{free}}[\delta_{\mb{a}}], \psi}
= \State{\delta_{\mb{a}} \rhd \psi}
= \State{\delta_{\mb{a}} \uplus \psi}.
\]
In several variables, $\Phi$ is not onto all the states with mean zero and identity covariance, for example it follows from the proof of Theorem~6 of \cite{AnsAppell3} that most of the free Meixner states are not in its image. Therefore in this case, $\State{\rho, \cdot}$ is never onto. In one variable, by the arguments used in the proof of Corollary~\ref{Cor:cpd-one-variable}, $\Phi$ is onto. Finally,
\[
\State{\cdot, \delta_{\mb{a}}} = \mf{B}_{\mb{a}, 0}[\State{\cdot, \delta_0}] = \mf{B}_{\mb{a}, 0} \circ \mf{B}^{-1}
\]
are all bijections from freely infinitely divisible states onto all states.
\end{proof}

\begin{Thm}
\label{Thm:Evolution}
Let $\rho$ be a freely infinitely divisible state with the free convolution semigroup $\set{\rho_t}$.
\begin{enumerate}
\item
\[
\State{\rho^{\boxplus t} \boxplus \delta_{\mb{a}}, \psi} = \State{\rho, \psi}^{\uplus t} \uplus \delta_{\mb{a}}.
\]
\item
\[
\State{\rho, \psi \boxplus \rho_t} = \mf{B}_t[\State{\rho, \psi}].
\]
Equivalently, if $R^{\phi, \psi} = R^\rho$, then $R^{\mf{B}_t[\phi], \psi \boxplus \rho_t} = R^{\phi, \psi} = R^\rho$ does not depend on $t$. More generally,
\[
\State{\rho, \psi \boxplus \rho^{\boxplus t} \boxplus \delta_{\mb{a}}}
= \mf{B}_{\mb{a}, t}[\State{\rho, \psi}]
= \mf{B}_t[\State{\rho, \psi \boxplus \delta_\mb{a}}]
= \mf{B}_{\mb{a}, 0}[\State{\rho, \psi \boxplus \rho_t}].
\]
\end{enumerate}
\end{Thm}

\noindent
We will present two proofs of this theorem. The combinatorial proof has all the details and works for general functionals. On the other hand, the operator-representation proof in the next section may be more illuminating.

\begin{proof}[Combinatorial proof]
Part (a) follows by exactly the same method as Lemma~\ref{Lemma:boxplus-uplus}.

\br
For part (b), on one hand, by expanding the defining relation~\eqref{eta-M-R},
\[
\begin{split}
& \eta^{\State{\rho, \psi \boxplus \rho_t}}[x_1, x_2, \ldots, x_n] \\
&\qquad = \sum_{\substack{\Lambda \subset \set{1, \ldots, n} \\ \Lambda = \set{u(0) = 1, u(1), u(2), \ldots, u(k) = n}}} \Cum{\rho}{x_i : i \in \Lambda} \prod_{j=1}^k (\psi \boxplus \rho_t)[x_i: u(j-1)+1 \leq i \leq u(j)-1].
\end{split}
\]
Combining this with the definition~\eqref{Free-convolution} of free convolution and expansion~\eqref{R-cumulant-moment}, we get
\begin{multline}
\label{Phi-psi-rho}
\eta^{\State{\rho, \psi \boxplus \rho_t}}[x_1, x_2, \ldots, x_n] \\
= \sum_{\pi \in \NC'(n)} \sum_{V: B^o(\pi) \in V \subset \pi} t^{\abs{V} - 1} \prod_{B \in V} \Cum{\rho}{x_i: i \in B} \prod_{C \in \pi \backslash V} \psi[x_i: i \in C].
\end{multline}
On the other hand,
\[
\begin{split}
& \eta^{\State{\rho, \psi}}[x_1, x_2, \ldots, x_n] \\
&\qquad = \sum_{\substack{\Lambda \subset \set{1, \ldots, n} \\ \Lambda = \set{u(0) = 1, u(1), u(2), \ldots, u(k) = n}}} \Cum{\rho}{x_i : i \in \Lambda} \prod_{j=1}^k \psi[x_i: u(j-1)+1 \leq i \leq u(j)-1] \\
&\qquad = \sum_{\sigma \in \NC'(n)} \Cum{\rho}{x_i: i \in B^o(\sigma)} \prod_{C \in \Inner(\pi)} \Cum{\psi}{x_i: i \in C}.
\end{split}
\]
Also, by Remark 4.4 of \cite{Belinschi-Nica-Free-BM},
\[
\eta^{\mf{B}_t[\phi]}[x_1, x_2, \ldots, x_n]
= \sum_{\omega \in \NC'(n)} t^{\abs{\omega} - 1} \prod_{B \in \omega} \eta^\phi[x_i: i \in B].
\]
Thus
\begin{equation}
\label{B-rho-psi}
\begin{split}
& \eta^{\mf{B}_t[\State{\rho, \psi}]}[x_1, x_2, \ldots, x_n] \\
&\qquad = \sum_{\omega \in \NC'(n)} t^{\abs{\omega} - 1} \prod_{B \in \omega} \sum_{\sigma \in \NC'(B)} \Cum{\rho}{x_i: i \in B^o(\sigma)} \prod_{C \in \Inner(\pi)} \Cum{\psi}{x_i: i \in C}.
\end{split}
\end{equation}
Clearly every term of the sum \eqref{B-rho-psi} appears in the sum \eqref{Phi-psi-rho} for some pair $(\pi, V)$. It remains to show the converse, namely that each such pair corresponds to the unique collection
\[
\set{\omega = (B_1, B_2, \ldots, B_k) \in \NC'(n), \set{\sigma_i \in \NC'(B_i)}}
\]
with
\begin{equation}
\label{pi-sigma}
i \stackrel{\pi}{\sim} j \Leftrightarrow i \stackrel{\sigma_s}{\sim} j \text{ for some } s
\end{equation}
and
\begin{equation}
\label{V-union}
V = \cup_{i=1}^k B_o(\sigma_i).
\end{equation}
This correspondence is very closely related to Lemma 3 of \cite{AnsFree-Meixner} and Remark 6.3 of \cite{Belinschi-Nica-Free-BM}. Namely, define $\omega$ as follows: for any $j$, let $j \stackrel{\omega}{\sim} i$ for $i$ the largest element with the property that
\[
i, i' \in B \in V, \quad i \leq j < i'.
\]
Note that since $B^o(\pi) \in V$, such an $i$ always exists. Clearly $\omega \in \NC'(n)$ and $\pi \leq \omega$, so we can define each $\sigma_s$ via equation \eqref{pi-sigma}. For each $B \in \omega$, $\min(B)$ and $\max(B)$ lie in the same class of $\pi$ which in fact belongs to $V$. Relation \eqref{V-union} and $\sigma_i \in \NC'(B_i)$ follow.

\br
For the second equation in part (b), it suffices to show that
\[
\mf{B}_{\mb{a}, 0}[\State{\rho, \psi}] = \State{\rho, \psi \boxplus \delta_\mb{a}},
\]
which follows by very similar methods using representation~\eqref{Formula-B_a}.
\end{proof}

\begin{Prop}
\label{Prop:Meixner-characterization}
Let $\rho, \psi$ be freely infinitely divisible distributions, so that $\phi = \State{\rho, \psi}$, by definition, is a c-freely infinitely divisible distribution. Suppose $\phi$ is a free Meixner distribution. Then all the distributions in its semigroup
\[
\phi(t) = \State{\rho^{\boxplus t}, \psi^{\boxplus t}}
\]
from Remark~\ref{Remark:Schoenberg} are (un-normalized) free Meixner if and only if $\rho$ is free Meixner and $\psi = \rho^{\boxplus (1+s)} \boxplus \delta_{\alpha}$. Thus, while the Boolean ($\psi = \delta_0$) and free ($\psi = \rho$) evolutions preserve the Meixner class, general c-free evolution does not.
\end{Prop}

\begin{proof}
If $\rho = \mu_{b,c}$ is freely infinitely divisible, so that $c \geq 0$, then
\[
\phi(t) = \State{\mu_{b,c}^{\boxplus t}, \mu_{b,c}^{\boxplus (1+s) t} \boxplus \delta_{\alpha t}}
= \State{\mu_{b,c}, \mu_{b,c}^{\boxplus (1+s) t} \boxplus \delta_{\alpha t}}^{\uplus t}
= \mu_{b + \alpha t, -1 + c + (1+s) t}^{\uplus t}
\]
is a free Meixner distribution, with $-1 + c + (1+s) t \geq -1$. For the converse, suppose that all $\phi(t)$ are free Meixner. Normalize $\phi$ to have mean zero and variance $1$. Then
\begin{equation}
\label{PDE-one-variable}
D^2 \eta^{\phi(t)}(z) = t + b(t) D \eta^{\phi(t)} (z) + c(t) \Bigl( D \eta^{\phi(t)}(z) \Bigr)^2.
\end{equation}
To simplify notation, we write $x_{u(1)}, \ldots, x_{u(n)}$, although in the single-variable context all of these are equal to $x$. We know that
\[
\begin{split}
\eta^{\phi}[x_{u(1)}, x_{u(2)}, \ldots, x_{u(n)}]
& = \sum_{\pi \in \NC'(n)} \Cum{\rho}{x_i: i \in B^o(\pi)} \prod_{C \in \Inner(\pi)} \Cum{\psi}{x_i: i \in C} \\
& = \Cum{\rho}{x_{u(1)}, x_{u(2)}, \ldots, x_{u(n)}} + \text{ products of lower order terms}.
\end{split}
\]
Therefore given $\phi$, each $\Cum{\rho}{\cdot}$ is uniquely determined by the lower order $\Cum{\psi}{\cdot}$. Thus
\[
\eta^{\phi(t)}[x_{u(1)}, x_{u(2)}, \ldots, x_{u(n)}] = \sum_{\pi \in \NC'(n)} t^{\abs{\pi}} \Cum{\rho}{x_i: i \in B^o(\pi)} \prod_{C \in \Inner(\pi)} \Cum{\psi}{x_i: i \in C}.
\]
can be expressed in terms of $\eta^{\phi}[\cdot]$ and $\Cum{\psi}{\cdot}$ of order $n$ and lower. On the other hand, from equation~\eqref{PDE-one-variable} we get
\[
\begin{split}
\eta^{\phi(t)}[x_j, x_i, x_{u(1)}, \ldots, x_{u(n)}]
& = b(t) \eta^{\phi(t)}[x_k, x_{u(1)}, \ldots, x_{u(n)}] \\
&\quad + \sum_{s=0}^n c(t) \eta^{\phi(t)}[x_i, x_{u(1)}, \ldots, x_{u(s)}] \eta^{\phi(t)}[x_j,x_{u(s+1)}, \ldots, x_{u(n)}].
\end{split}
\]
These equations for $n = 1, 2$ determine $b(t)$ and $c(t)$ uniquely in terms of $\eta^{\phi}[\cdot]$ and $\Cum{\psi}{\cdot}$; in fact,
\[
b(t) = b(1) + (t-1) \Cum{\psi}{x}, \qquad c(t) = \frac{1}{t} \left( c(1) + (t-1) \Cum{\psi}{x,x} \right).
\]
Therefore these equations for larger $n$ determine all the $\psi$-cumulants in terms of $\eta^{\phi}[\cdot]$ and the $\psi$-cumulants of order $1$ and $2$.
If the mean of $\psi$ is $\alpha$ and variance $1+s$, then $\psi = \rho^{\boxplus (1+s)} \boxplus \delta_{\alpha}$ has the correct first two cumulants, and therefore is the correct state. Finally,
\[
\phi = \State{\rho, \rho^{\boxplus (1+s)} \boxplus \alpha} = \mf{B}_{\alpha, s}[\rho]
\]
and so $\rho = \mf{B}_{-\alpha, -s}[\phi]$; since $\rho$ is a state, this expression is well defined. Since $\phi$ is free Meixner, it follows that $\rho$ is a free Meixner distribution.
\end{proof}

\begin{Remark}
\label{Remark:Lenczewski}
In addition to monotone convolution, another operation related to $\Phi$ is the orthogonal convolution $\vdash$ of Lenczewski. In fact, in the single variable case,
\[
\State{\mf{B}[\tau], \psi} = \tau \vdash \psi.
\]
Some results in this paper can be reformulated as multivariate versions of the results in \cite{Lenczewski-Decompositions-convolution}. For example, our equation~\eqref{Phi-monotone} in Lemma~\ref{Lemma:Monotone-Phi} can be extended to
\[
\State{\tau \rhd \psi} = \State{\mf{B}[\State{\tau}], \psi} = \State{\tau} \vdash \psi,
\]
closely related to Corollary~6.4 of \cite{Lenczewski-Decompositions-convolution}. In Theorem~\ref{Thm:Phi-properties}, the positivity of $\Phi$ in part (b) corresponds to the positivity of the orthogonal convolution; this also explains why the first argument of $\Phi$ is naturally taken to be freely infinitely divisible. Parts (d) and (e) of that theorem imply that
\[
\mu_{b, -1} \vdash \psi = \mf{B}_{b,0}[\State{\psi}]
\]
and
\[
\tau \vdash \delta_a = \mf{B}_{a,0}[\tau],
\]
which are Examples 6.2 and 6.1 of \cite{Lenczewski-Decompositions-convolution}, respectively. Similarly, Theorem~\ref{Thm:Evolution} can be re-formulated in terms of the orthogonal convolution. In fact, some results in \cite{Lenczewski-Decompositions-convolution} are obtained by complex-analytic methods which apply to measures with possibly infinite moments, and can be used to obtain extensions of our results in the single-variable context.
\end{Remark}

\subsection{Operator models}
\label{Subsec:Operator-models}

\begin{Lemma}
\label{Lemma:Basic-operator-representations}
For a Hilbert space $\mc{H}$, by its Boolean Fock space we mean the Hilbert space $\mf{C} \Omega \oplus \mc{H}$ and by its full Fock space the Hilbert space $\mc{F}(\mc{H}) = \mf{C} \Omega \oplus \bigoplus_{n=1}^\infty \mc{H}^{\otimes n}$. For $\zeta \in \mc{H}$, its Boolean creation and annihilation operators on the Boolean Fock space are
\begin{align*}
a_\zeta^{b,+}(\eps) &= \ip{\Omega}{\eps} \zeta, \\
a_\zeta^{b,-}(\eps) &= \ip{\zeta}{\eps} \Omega,
\end{align*}
and its free creation and annihilation operators on the full Fock space are
\begin{align*}
a_{\zeta}^{f,+}(\eps_1 \otimes \ldots \otimes \eps_n) &= \zeta \otimes \eps_1 \otimes \ldots \otimes \eps_n, \\
a_{\zeta}^{f,-}(\eps_1 \otimes \ldots \otimes \eps_n) &= \ip{\zeta}{\eps_1} \eps_2 \otimes \ldots \otimes \eps_n.
\end{align*}
For $H \in \mc{L}(\mc{H})$, it acts on the Boolean Fock space by $H \Omega = 0$, and its gauge operator on the full Fock space is
\[
p(H) (\eps_1 \otimes \ldots \otimes \eps_n) = (H \eps_1) \otimes \eps_2 \otimes \ldots \otimes \eps_n.
\]
Finally, denote by $P_\Omega$ the projection on $\Omega$.
\begin{enumerate}
\item
A state $\psi$ corresponds to a collection of data
\[
(\mc{K}, \xi \in \mc{K}, K_i \in \mc{L}(\mc{K})),
\]
where $\mc{K}$ is a Hilbert space, $\xi \in \mc{K}$ a unit vector, and $\set{K_1, K_2, \ldots, K_d} \in \mc{L}(\mc{K})$ a $d$-tuple of symmetric operators with a common invariant dense domain containing $\xi$, via
\[
\psi[x_{\vec{u}}] = \ip{\xi}{K_{\vec{u}} \xi}.
\]
Conversely, any such collection always gives a state. (We will omit the last comment and the conditions on the operators in subsequent constructions.)
\item
A state $\phi$ also corresponds to a collection of data
\[
(\mc{K}, \eps_i \in \mc{K}, S_i \in \mc{L}(\mc{K}), \alpha_i \in \mf{R})
\]
as follows: on the Boolean Fock space $\mf{C} \Omega \oplus \mc{K}$,
\[
\phi[x_{\vec{u}}] = \ip{\Omega}{\Bigl(a_{\eps_i}^{b,+} + a_{\eps_i}^{b,-} + S_i + \alpha_i P_\Omega \Bigr)_{\vec{u}} \Omega}.
\]
Here $\ip{\eps_i}{S_{\vec{u}} \eps_j}$ are the Boolean cumulants of $\phi$.
\item
A conditionally positive definite functional $\mu$ corresponds to a collection of data
\[
(\mc{H}, \zeta_i \in \mc{H}, H_i \in \mc{L}(\mc{H}), \lambda_i \in \mf{R})
\]
via
\[
\mu \left[ x_i x_{\vec{u}} x_j \right] = \ip{\zeta_i}{ H_{\vec{u}} \zeta_j}, \qquad \mu[x_i] = \lambda_i.
\]
\item
A freely infinitely divisible state $\rho$ corresponds to a collection of data
\[
(\mc{H}, \zeta_i \in \mc{H}, H_i \in \mc{L}(\mc{H}), \lambda_i \in \mf{R})
\]
as follows: on the full Fock space $\mc{F}(\mc{H})$,
\[
\rho[x_{\vec{u}}] = \ip{\Omega}{\Bigl(a_{\zeta_i}^{f,+} + a_{\zeta_i}^{f,-} + p(H_i) + \lambda_i I \Bigr)_{\vec{u}} \Omega}.
\]
Here $\ip{\zeta_i}{H_{\vec{u}} \zeta_j}$ are the free cumulants of $\rho$.
\end{enumerate}
\end{Lemma}

\begin{proof}
Part (a) is standard; briefly, take $\mc{K}$ to be the completion of the quotient of $\mf{C} \langle \mb{x} \rangle$ with respect to the $\psi$-norm, $\xi$ to be the vector image of $1$, and $K_i$ to be the operator image of $x_i$. Part (b) is Proposition~15 of \cite{AnsBoolean}. Part (c) is also standard, see Proposition 3.2 of \cite{SchurCondPos}; briefly, take $\mc{H}$ to be the completion of the quotient of
\[
\set{P \in \mf{C} \langle \mb{x} \rangle | P(0) = 0}
\]
with respect to the $\mu$-norm, $\zeta_i$ to be the vector image of $x_i$, and $K_i$ to be the operator image of $x_i$. Finally, since a state $\rho$ is freely infinitely divisible if and only if $R^\rho$ is conditionally positive definite, part (d) follows from part (c).
\end{proof}

\begin{proof}[Proof of Theorem~\ref{Thm:c-free-representation}]
For $\psi, \mu$ represented as in Lemma~\ref{Lemma:Basic-operator-representations}, on the Hilbert space
\[
\mc{K} \otimes \mc{H},
\]
consider the operators
\[
K_i \otimes I + P_\xi \otimes H_i,
\]
where $P_\xi$ is the projection onto $\xi$ in $\mc{K}$. We will show that
\begin{equation}
\label{eta-cpd}
\eta \left[ x_i \prod_{s=1}^n x_{u(s)} x_j \right] = \ip{\xi \otimes \zeta_i}{ \prod_{s=1}^n \Bigl(K_{u(s)} \otimes I + P_\xi \otimes H_{u(s)} \Bigr) (\xi \otimes \zeta_j)}.
\end{equation}
The operators $K_i \otimes I + P_\xi \otimes H_i$ are symmetric, therefore it will follow that $\eta$ is conditionally positive definite.

\br
To prove \eqref{eta-cpd}, we note that
\[
\begin{split}
& \ip{\xi \otimes \zeta_i}{ \prod_{s=1}^k \Bigl(K_{u(s)} \otimes I + P_\xi \otimes H_{u(s)} \Bigr) (\xi \otimes \zeta_j)} \\
&\quad = \sum_{\substack{\Lambda \subset \set{1, 2, \ldots, n} \\ \Lambda = \set{v(1), v(2), \ldots, v(l)}}} \ip{\xi}{K_{u(1)} \ldots K_{u(v(1)-1)} P_\xi K_{u(v(1)+1)} \ldots K_{u(v(2)-1)} P_\xi \ldots K_{u(n)} \xi} \\
&\qquad\qquad\qquad\qquad\qquad \ip{\zeta_i}{I \ldots I H_{u(v(1))} I \ldots I H_{u(v(2))} I \ldots I H_{u(v(l))} I \ldots I \zeta_j} \\
&\quad = \sum_{\substack{\Lambda \subset \set{1, 2, \ldots, n} \\ \Lambda = \set{v(1), v(2), \ldots, v(l)}}}
\ip{\zeta_i}{\prod_{s=1}^l H_{u(v(s))} \zeta_j} \prod_{r=0}^l \ip{\xi}{\prod_{s = v(r)+1}^{v(r+1)-1} K_{u(s)} \xi} \\
&\quad = \sum_{\substack{\Lambda \subset \set{1, 2, \ldots, n} \\ \Lambda = \set{v(1), v(2), \ldots, v(l)}}} \mu \left[ x_i \prod_{s=1}^l x_{u(v(s))} x_j \right] \prod_{r=0}^l \psi \left[ \prod_{s = v(r)+1}^{v(r+1)-1} x_{u(s)} \right].
\end{split}
\]
It remains to note that
\[
\begin{split}
& 1 + \sum_{n=1}^\infty \sum_{\abs{\vec{u}} = n} \sum_{\substack{\Lambda \subset \set{1, 2, \ldots, n} \\ \Lambda = \set{1, v(1), v(2), \ldots, v(l), n}}} \mu \left[ x_{u(1)} \prod_{s=1}^l x_{u(v(s))} x_{u(n)} \right] \prod_{r=0}^l \psi \left[ \prod_{s = v(r)+1}^{v(r+1)-1} x_{u(s)} \right] w_{\vec{u}} \\
&\quad = 1 + \sum_{n=1}^\infty \sum_{\abs{\vec{u}} = n} \mu \left[ x_{u(1)} x_{u(2)} \ldots x_{u(n)} \right] w_{u(1)} (1 + M^\psi(\mb{w})) w_{u(2)} (1 + M^\psi(\mb{w})) w_{u(3)} \\
&\qquad\qquad\qquad\qquad \ldots (1 + M^\psi(\mb{w})) w_{u(n)} \\
&\quad = 1 + (1 + M^\psi(\mb{w}))^{-1} M^\mu((1 + M^\psi(\mb{w})) \mb{w})
= 1 + M^\eta(\mb{w}). \qedhere
\end{split}
\]
\end{proof}

\begin{Remark}[Relation to monotone probability]
\label{Remark:Monotone}
Lenczewski used a similar construction in Boo\-le\-an probability, and Franz and Muraki \cite{Franz-Muraki-Markov-monotone} in monotone probability: if $\psi$ is a joint distribution of the operators $\set{K_i}$ with respect to the vector state of $\xi$, and $\phi$ is the joint distribution of the operators $\set{H_i}$ with respect to the vector state of $\zeta$, then the joint distribution of the operators $\set{K_i \otimes I + P_\xi \otimes H_i}$ with respect to the vector state of $\xi \otimes \zeta$ is the monotone convolution $\phi \rhd \psi$. This provides an operator representation proof of Lemma~\ref{Lemma:Monotone-Phi}.
\end{Remark}

\begin{Lemma}
\label{Lemma:B-Phi-representations}
We use the notation of Lemma~\ref{Lemma:Basic-operator-representations}.
\begin{enumerate}
\item
For a freely infinitely divisible state $\rho$ and a state $\psi$, the state $\State{\rho, \psi}$ is represented on the Boolean Fock space
\[
\mf{C} \Omega \oplus (\mc{K} \otimes \mc{H})
\]
as
\[
\State{\rho, \psi}[x_{\vec{u}}] = \ip{\Omega}{\Bigl(a_{\xi \otimes \zeta_i}^{b,+} + a_{\xi \otimes \zeta_i}^{b,-} + K_i \otimes I + P_\xi \otimes H_i + \lambda_i P_\Omega \Bigr)_{\vec{u}} \Omega}.
\]
\item
For a state $\phi$, the freely infinitely divisible state $\mf{B}[\phi]$ is represented on the full Fock space $\mc{F}(\mc{K})$ as
\[
\mf{B}[\phi][x_{\vec{u}}] = \ip{\Omega}{\Bigl(a_{\eps_i}^{f,+} + a_{\eps_i}^{f,-} + p(S_i) + \alpha_i I \Bigr)_{\vec{u}} \Omega}.
\]
\item
For $\rho$, $\psi$ as above, the state $\psi \boxplus \rho$ is represented on the space
\[
\mc{F}(\mc{K} \otimes \mc{H}) \otimes \mc{K} \simeq \mc{K} \oplus \Bigl( \bigoplus_{n=1}^\infty (\mc{K} \otimes \mc{H})^{\otimes n} \otimes \mc{K} \Bigr)
\]
as
\[
(\psi \boxplus \rho)[x_{\vec{u}}] = \ip{\xi}{ \Bigl(a_{\xi \otimes \zeta_i}^{f,+} \otimes I + a_{\xi \otimes \zeta_i}^{f,-} \otimes I + p(K_i \otimes I) \otimes I + p(P_\xi \otimes H_i) \otimes I + \lambda_i I \Bigr)_{\vec{u}} \xi}.
\]
\end{enumerate}
\end{Lemma}

\begin{proof}
Part (a) follows by combining Theorem~\ref{Thm:c-free-representation} with part (b) of Lemma~\ref{Lemma:Basic-operator-representations}. Part (b) follows from the definition that
\[
\CumFun{\mf{B}[\phi]}{\mb{z}} = \eta^{\phi}(\mb{z})
\]
and parts (b), (d) of Lemma~\ref{Lemma:Basic-operator-representations}. Part (c) follows from the fact that
\[
(\psi \boxplus \rho)[x_1, x_2, \ldots, x_n] = \sum_{\substack{B_1, B_2, \ldots, B_k \subset \set{1, \ldots, n} \\ B_i \cap B_j = \emptyset \text{ for } i \neq j}} \prod_{i=1}^k \Cum{\rho}{x_j: j \in B_i} \prod_{C \in (B_1, B_2, \ldots, B_k)^c} \psi[x_j: j \in C],
\]
where
\[
(B_1, B_2, \ldots, B_k)^c = (C_1, C_2, \ldots, C_l)
\]
is the smallest collection of disjoint subsets of $\set{1, \ldots, n}$ such that
\[
(B_1, \ldots, B_k, C_1, \ldots, C_l) \in \NC(n). \qedhere
\]
\end{proof}

\begin{Remark}[Operator representation proof of Theorem~\ref{Thm:Evolution}]
Combining parts (a), (b) of the preceding Lemma, the freely infinitely divisible state $\mf{B}[\State{\rho, \psi}]$ is represented on the full Fock space
\[
\mc{F}(\mc{K} \otimes \mc{H}) = \mf{C} \Omega \oplus \bigoplus_{n=1}^\infty (\mc{K} \otimes \mc{H})^{\otimes n}
\]
as
\[
\mf{B}[\State{\rho, \psi}][x_{\vec{u}}] = \ip{\Omega}{\Bigl(a_{\xi \otimes \zeta_i}^{f,+} + a_{\xi \otimes \zeta_i}^{f,-} + p(K_i \otimes I + P_\xi \otimes H_i) + \lambda_i I \Bigr)_{\vec{u}} \Omega}.
\]

\br
On the other hand, combining parts (a), (c), the state $\State{\rho, \psi \boxplus \rho}$ is represented on the space
\[
\mf{C} \Omega \oplus \mc{F}(\mc{K} \otimes \mc{H}) \otimes \mc{K} \otimes \mc{H} \simeq \mc{F}(\mc{K} \otimes \mc{H})
\]
as
\[
\begin{split}
\State{\rho, \psi \boxplus \rho}[x_{\vec{u}}] & = \Bigl\langle \Omega, \Bigl(a_{\xi \otimes \zeta_i}^{b,+} + a_{\xi \otimes \zeta_i}^{b,-} \\
&\qquad + \bigl(a_{\xi \otimes \zeta_i}^{f,+} \otimes I + a_{\xi \otimes \zeta_i}^{f,-} \otimes I + p(K_i \otimes I) \otimes I + p(P_\xi \otimes H_i) \otimes I + \lambda_i I \bigr) \otimes I \\
&\qquad + I \otimes P_\xi \otimes H_i + \lambda_i P_\Omega \Bigr)_{\vec{u}} \Omega \Bigr\rangle \\
& = \ip{\Omega}{\Bigl(a_{\xi \otimes \zeta_i}^{f,+} + a_{\xi \otimes \zeta_i}^{f,-} + p(K_i \otimes I) + p(P_\xi \otimes H_i) + \lambda_i I \Bigr)_{\vec{u}} \Omega}.
\end{split}
\]
Thus $\mf{B}[\State{\rho, \psi}] = \State{\rho, \psi \boxplus \rho}$. The more general equation involving $\mf{B}_t$ follows from this by using
\[
\mf{B}_t[\phi] = \mf{B}[\phi^{\uplus t}]^{\uplus (1/t)},
\]
which in turn can be deduced from the semigroup property of $\set{\mf{B}_t}$.

\br
There are proofs by similar methods of other results in the paper, such as the remaining parts of Theorem~\ref{Thm:Evolution} and Lemma~\ref{Lemma:B_t-semigroup}; they are left to an interested reader.
\end{Remark}


\begin{thebibliography}{VDN92}

\bibitem[Ans07a]{AnsFree-Meixner}
Michael Anshelevich, \emph{Free {M}eixner states}, Comm. Math. Phys. \textbf{276} (2007),
  no.~3, 863--899. \MR{MR2350440}

\bibitem[Ans07b]{AnsBoolean}
\bysame, \emph{{A}ppell polynomials and their relatives {II}.
  {B}oolean theory}, \texttt{arXiv:0712.4185 [math.OA]}, 2007.

\bibitem[Ans08a]{AnsMulti-Sheffer}
\bysame, \emph{Orthogonal polynomials with a resolvent-type generating
  function}, \texttt{arXiv:math/0410482 [math.CO]}, To be published by the
  Transactions of the AMS, 2008.

\bibitem[Ans08b]{AnsAppell3}
\bysame, \emph{{A}ppell polynomials and their relatives {III}. {C}onditionally
  free theory.}, \texttt{arXiv:0803.4279 [math.OA]}, 2008.

\bibitem[BN07a]{Belinschi-Nica-B_t}
Serban~T. Belinschi and Alexandru Nica, \emph{On a remarkable semigroup of homomorphisms with respect to free
  multiplicative convolution}, \texttt{arXiv:math/0703295 [math.OA]}, 2007.

\bibitem[BN07b]{Belinschi-Nica-Free-BM}
\bysame, \emph{Free {B}rownian motion and
  evolution towards $\boxplus$-infinite divisibility for $k$-tuples},
  \texttt{arXiv:0711.3787 [math.OA]}, 2007.

\bibitem[BN08]{Belinschi-Nica-Eta}
\bysame, \emph{{$\eta$}-series and a {B}oolean {B}ercovici-{P}ata bijection for
  bounded {$k$}-tuples}, Adv. Math. \textbf{217} (2008), no.~1, 1--41.
  \MR{MR2357321}

\bibitem[BB06]{Boz-Bryc}
Marek Bo{\.z}ejko and W{\l}odzimierz Bryc, \emph{On a class of free {L}\'evy
  laws related to a regression problem}, J. Funct. Anal. \textbf{236} (2006),
  no.~1, 59--77. \MR{MR2227129 (2007a:46071)}

\bibitem[BB08]{Boz-Bryc-Two-states}
\bysame, \emph{A quadratic regression problem for two-state algebras with
  application to the central limit theorem}, \texttt{arXiv:0802.0266
  [math.OA]}, 2008.

\bibitem[BLS96]{BLS96}
Marek Bo{\.z}ejko, Michael Leinert, and Roland Speicher, \emph{Convolution and
  limit theorems for conditionally free random variables}, Pacific J. Math.
  \textbf{175} (1996), no.~2, 357--388. \MR{MR1432836 (98j:46069)}

\bibitem[BW01]{Boz-Wys}
Marek Bo{\.z}ejko and Janusz Wysocza{\'n}ski, \emph{Remarks on
  {$t$}-transformations of measures and convolutions}, Ann. Inst. H. Poincar\'e
  Probab. Statist. \textbf{37} (2001), no.~6, 737--761. \MR{MR1863276
  (2002i:60005)}

\bibitem[Dur91]{Dur}
Richard Durrett, \emph{Probability}, The Wadsworth \& Brooks/Cole
  Statistics/Probability Series, Wadsworth \& Brooks/Cole Advanced Books \&
  Software, Pacific Grove, CA, 1991.

\bibitem[Fra03]{Franz-Unification}
Uwe Franz, \emph{Unification of {B}oolean, monotone, anti-monotone, and tensor
  independence and {L}\'evy processes}, Math. Z. \textbf{243} (2003), no.~4,
  779--816. \MR{MR1974583 (2004f:46077)}

\bibitem[Fra05]{Franz-Multiplicative-monotone}
\bysame, \emph{Multiplicative monotone convolutions},
  \texttt{arXiv:math/0503602 [math.PR]}, 2005.

\bibitem[FM05]{Franz-Muraki-Markov-monotone}
Uwe Franz and Naofumi Muraki, \emph{Markov property of monotone {L}\'evy
  processes}, Infinite dimensional harmonic analysis III, World Sci. Publ.,
  Hackensack, NJ, 2005, pp.~37--57. \MR{MR2230621 (2007e:60033)}

\bibitem[Kry07]{Krystek-Conditional}
Anna~Dorota Krystek, \emph{Infinite divisibility for the conditionally free
  convolution}, Infin. Dimens. Anal. Quantum Probab. Relat. Top. \textbf{10}
  (2007), no.~4, 499--522.

\bibitem[KW05]{Kry-Woj-Associative}
Anna~Dorota Krystek and {\L}ukasz~Jan Wojakowski, \emph{Associative
  convolutions arising from conditionally free convolution}, Infin. Dimens.
  Anal. Quantum Probab. Relat. Top. \textbf{8} (2005), no.~3, 515--545.
  \MR{MR2172313 (2006g:46103)}

\bibitem[KY04]{Krystek-Yoshida-t}
Anna Krystek and Hiroaki Yoshida, \emph{Generalized {$t$}-transformations of
  probability measures and deformed convolutions}, Probab. Math. Statist.
  \textbf{24} (2004), no.~1, Acta Univ. Wratislav. No. 2646, 97--119.
  \MR{MR2108159 (2006i:46093)}

\bibitem[LL60]{Laha-Lukacs}
R.~G. Laha and E.~Lukacs, \emph{On a problem connected with quadratic
  regression}, Biometrika \textbf{47} (1960), 335--343. \MR{MR0121922 (22
  \#12649)}

\bibitem[Len07]{Lenczewski-Decompositions-convolution}
Romuald Lenczewski, \emph{Decompositions of the free additive convolution}, J.
  Funct. Anal. \textbf{246} (2007), no.~2, 330--365. \MR{MR2321046}

\bibitem[NS06]{Nica-Speicher-book}
Alexandru Nica and Roland Speicher, \emph{Lectures on the combinatorics of free
  probability}, London Mathematical Society Lecture Note Series, vol. 335,
  Cambridge University Press, Cambridge, 2006. \MR{MR2266879}

\bibitem[Ora02]{Oravecz-Fermi}
Ferenc Oravecz, \emph{Fermi convolution}, Infin. Dimens. Anal. Quantum Probab.
  Relat. Top. \textbf{5} (2002), no.~2, 235--242. \MR{MR1914835 (2003e:46112)}

\bibitem[Sch91]{SchurCondPos}
Michael Sch{\"u}rmann, \emph{Quantum stochastic processes with independent
  additive increments}, J. Multivariate Anal. \textbf{38} (1991), no.~1,
  15--35. \MR{92k:46113}

\bibitem[VDN92]{VDN}
D.~V. Voiculescu, K.~J. Dykema, and A.~Nica, \emph{Free random variables}, CRM
  Monograph Series, vol.~1, American Mathematical Society, Providence, RI,
  1992. \MR{MR1217253 (94c:46133)}

\end{thebibliography}

\providecommand{\bysame}{\leavevmode\hbox to3em{\hrulefill}\thinspace}
\providecommand{\MR}{\relax\ifhmode\unskip\space\fi MR }
\providecommand{\MRhref}[2]{%
  \href{http://www.ams.org/mathscinet-getitem?mr=#1}{#2}
}
\providecommand{\href}[2]{#2}

\end{document}